\documentclass[
	11pt,
	leqno, 
]{amsart}

\usepackage{amsthm, amsmath, amsfonts, amssymb, graphicx, longtable,mathrsfs,mathtools,upref,tikz,physics}	
\usepackage[foot]{amsaddr}	
\usepackage[
	style=numeric-comp,
	sorting=none,
	sortcites=true,
	giveninits=true,
	hyperref]{biblatex}
\usepackage[linktoc=page]{hyperref}
\hypersetup{
	pdfauthor={Sally Ambrose, Evan Angelone, Jacob Chen, Daniel Ma, Arturo Ortiz San Miguel, Wraven Watanabe, Stephen Whitcomb, Shanghao Wu},
	colorlinks=true
}
\usepackage[noabbrev]{cleveref}

\renewcommand{\cref}[1]{\Cref{#1}} 
\usepackage[cal=rsfs,bb=txof]{mathalpha}
\usepackage[nodisplayskipstretch]{setspace}
\usepackage{enumitem}

\usepackage{tikz}
\usepackage{xparse}
\usepackage{graphicx}
\usepackage[dvipsnames]{xcolor}
\usepackage{subcaption}
\usepackage{manfnt}

\newcommand{\bK}{K}
\newcommand{\bQ}{Q}

\newcommand{\bN}{N}

\NewDocumentCommand{\drawchessboard}{ m m m}{%
  \begin{tikzpicture}[scale=0.6]
    \tikzset{
      chesspiece/.style={},  
      coordlabel/.style={}  
    }

    \foreach \x in {1,...,#2} {
      \foreach \y in {1,...,#2} {
        \pgfmathtruncatemacro{\iswhite}{mod(\x+\y,2)}
        \ifnum\iswhite=0
          \fill[lightgray] (\x-1,\y-1) rectangle (\x,\y); 
        \else
          \fill[white] (\x-1,\y-1) rectangle (\x,\y); 
        \fi
        \draw (\x-1,\y-1) rectangle (\x,\y); 
      }
    }

    \ifthenelse{\equal{#3}{true}}{
      \foreach \x in {1,...,#2} {
        \node[coordlabel, anchor=north] at (\x-0.5, -0.2) {\x};
      }
      \foreach \y in {1,...,#2} {
        \node[coordlabel, anchor=east] at (-0.2, \y-0.5) {\y};
      }
    }{}
    \begin{scope}[shift={(-0.5,-0.5)}]
      #1
    \end{scope}    
  \end{tikzpicture}}
\theoremstyle{definition}

\theoremstyle{plain}
	\newtheorem{theorem}{Theorem}[section]

\newtheorem{lemma}[theorem]{Lemma}

\theoremstyle{remark}
	\newtheorem*{remark}{Remark}
	
\title[Cops and robbers on chess graphs]{Cops and robbers on chess graphs}

\author[]{Sally Ambrose$^{*}$}
\address{Northeastern University, Boston, Massachusetts, USA}

\author[]{Evan Angelone$^{\dagger}$}
\email{\url{griggs.e@northeastern.edu}}

\author[]{Jacob Chen$^{*}$}

\author[]{Daniel Ma$^{*}$}

\author[]{Arturo Ortiz San Miguel$^{\dagger}$}
\email{\url{ortizsanmiguel.a@northeastern.edu}}

\author[]{Wraven Watanabe$^{*}$}

\author[]{Stephen Whitcomb$^{*}$}

\author[]{Shanghao Wu$^{*}$}

\thanks{This research was part of the Northeastern Summer Math Research
Program (NSMRP), which is an REU. The authors were partially supported by the Northeastern University College of Science, the Northeastern University Department of Mathematics, and NSF grant DMS-1645877.}
\thanks{$^{*}$Undergraduate mentee.}
\thanks{$^{\dagger}$PhD student mentor.}


\begin{document}
	\begin{abstract}
	Cops and robbers is a pursuit-evasion game played on graphs. We completely classify the cop numbers for $n \times n$ knight graphs and queen graphs. This completes the classification of the cop numbers for all $n \times n$ classical chess graphs. As a corollary, we resolve an open problem about the monotonicity of $c$($\mathcal{Q}_n$). Moreover, we introduce \emph{royal graphs}, a generalization of chess graphs for arbitrary piece movements, which models real-life movement constraints. We give results on the cop numbers for these families.
\end{abstract}
	\maketitle

	\section{Introduction}\label{sec:introduction}
	\subsection{Rules and notation}\label{sub:rules_and_notation}
	Cops and robbers is a perfect knowledge pursuit-evasion game inspired by the popular children's game of the same name. The game is played on a graph ~$G=(V,E)$. In cops and robbers, one player (the robber) attempts to evade capture by the other player (the cops) by taking turns moving along the edges of a graph.
	\medskip

	The rules of the game are as follows:
    \begin{enumerate}
	    \item A fixed number (say, $k$) of cops are each placed on a starting vertex. The robber then chooses a starting vertex.
	    \item During each turn, players can either move to adjacent vertices or choose not to move. Players alternate turns with the cops moving first.
	    \item If at least one of the~$k$ cops lands on the same vertex as the robber, then the cops win. Otherwise, if the robber can evade capture indefinitely, the robber wins.
    \end{enumerate} 
    If~$k$ cops can always win after a finite number of turns, regardless
    of the robber's strategy, we say~$G$ is \emph{$k$-cop win}. Clearly, if we let~$k = |V|$, the cops win. As such, we are interested in the smallest number $k$ of cops needed to win on a graph. Such $k$ is known as the \emph{cop number} of~$G$ and denoted by~$c(G)$. If a graph is 1-cop win, we call it \emph{cop-win}.
    \medskip

    The study of cops and robbers (including its many variants) and the corresponding cop number has been extensively studied across a variety of graphs. One classical example where graphs arise is the study of chess pieces and how they move. A \emph{chess graph} is a graph-theoretic representation of the movement of a chess piece on a board. The vertices of a chess graph represent the squares of a chessboard, with adjacencies determined by whether the corresponding piece can move between them. We denote the $n\times n$ knight and queen graphs as $\mathcal{N}_n$ and $\mathcal{Q}_n$, respectively. Here, we give cop numbers for knight and queen graphs. This completes the classification of the cop numbers for $n \times n$ classical chess graphs. 

     \begin{figure}[h!]
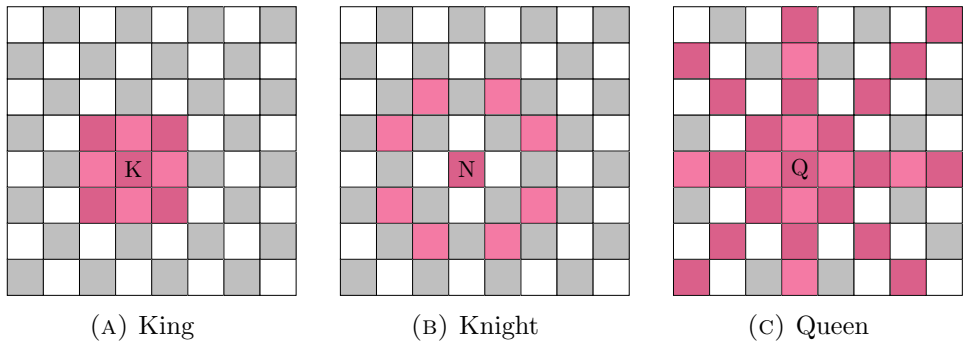

        \centering
        \begin{subfigure}[b]{0.3\textwidth}
        \centering
        \scalebox{0.8}{
        \drawchessboard{
          \foreach \dx in {-1,0,1} {
            \foreach \dy in {-1,0,1} {
              \pgfmathtruncatemacro{\newx}{4 + \dx}
              \pgfmathtruncatemacro{\newy}{4 + \dy}
                \ifthenelse{\newx > 0 \AND \newx < 9 \AND \newy > 0 \AND \newy < 9}{
              \node[fill=WildStrawberry, opacity=0.6, minimum size=.57cm, inner sep=0] at (\newx, \newy) {};
                }{}
            }
          }
          \node[chesspiece] at (4,4) {\bK};
        }{8}{false}
        }
        \caption{King}
        \end{subfigure}
        \hfill
        \begin{subfigure}[b]{0.3\textwidth}
        \centering
        \scalebox{.8}{
        \drawchessboard{
            \foreach \dx/\dy in {2/1, 2/-1, -2/1, -2/-1, 1/2, 1/-2, -1/2, -1/-2, 0/0} {
            \pgfmathtruncatemacro{\newx}{4 + \dx}
            \pgfmathtruncatemacro{\newy}{4 + \dy}
            \ifthenelse{\newx > 0 \AND \newx < 9 \AND \newy > 0 \AND \newy < 9}{
              \node[fill=WildStrawberry, opacity=0.6, minimum size=.57cm, inner sep=0] at (\newx, \newy) {};
            }{}
            }
            \node[chesspiece] at (4,4) {\bN};
        }{8}{false}
        }
        \caption{Knight}
        \end{subfigure}
        \hfill
        \begin{subfigure}[b]{0.3\textwidth}
        \centering
        \scalebox{.8}{
        \drawchessboard{
          \foreach \i in {1,2,3,5,6,7,8} {
            \node[fill=WildStrawberry, opacity=0.6, minimum size=.57cm, inner sep=0] at (\i, 4) {};
            \node[fill=WildStrawberry, opacity=0.6, minimum size=.57cm, inner sep=0] at (4, \i) {};
          };
          \foreach \offset in {1,2,3} {
            \node[fill=WildStrawberry, opacity=0.6, minimum size=.57cm, inner sep=0] at ({4+\offset}, {4+\offset}) {};
            \node[fill=WildStrawberry, opacity=0.6, minimum size=.57cm, inner sep=0] at ({4-\offset}, {4+\offset}) {};
            \node[fill=WildStrawberry, opacity=0.6, minimum size=.57cm, inner sep=0] at ({4+\offset}, {4-\offset}) {};
            \node[fill=WildStrawberry, opacity=0.6, minimum size=.57cm, inner sep=0] at ({4-\offset}, {4-\offset}) {};
          };
          \node[fill=WildStrawberry, opacity=0.6, minimum size=.57cm, inner sep=0] at (8,8) {};
          \node[fill=WildStrawberry, opacity=0.6, minimum size=.57cm, inner sep=0] at (4,4) {};
          \node[chesspiece] at (4,4) {\bQ};
        }{8}{false}
        }
        \caption{Queen}
        \end{subfigure}
        
        \caption{Chess graphs highlighting neighborhoods $N(v)$ of the vertex $v$ at (4,4).}
        \label{fig:chessgraphs}
    \end{figure}
    
    As such, we additionally consider further generalizations of chess graphs, known as \emph{royal graphs} and \emph{animal graphs}. These generalizations respectively extend the global and local movement of pieces to arbitrary directions. In applications, these graphs model movement constraints typical to machines or animals in their surroundings. Studying cops and robbers on royal and animal graphs provides a framework for modeling pursuit-evasion dynamics on structured environments, with potential applications in robotics and network navigation, where movement constraints mimic directional or geometric limitations of objects and their environments.
    \smallskip
	\subsection{Discussion and results}\label{sub:discussion_and_results}
	Cop numbers are a well-studied property for certain classes of graphs. For example, cop-win graphs are particularly well understood. A graph $G$ is cop-win if and only if $G$ is \emph{dismantlable} \cite{inbook}.

    Outside of special families of graphs, cop numbers are generally difficult (in fact, EXPTIME-complete \cite{kinnersley2015cops}) to compute for arbitrary graphs. As a result, there exist several general upper bounds on the cop number. One of the long-standing conjectures is Meyniel's conjecture, which states that for all graphs $G$ on $n$ vertices, $c(G) = O(\sqrt{n})$. In fact, it is still open to show that
            $c(G) = O(n^{1-\varepsilon}),$ for arbitrarily small $\varepsilon$. Furthermore, \cite{hosseini} showed that if Meyniel's conjecture holds for subcubic graphs, then the weak Meyniel's conjecture holds for $\varepsilon < 1/4$. The best known general bound, shown by Scott and Sudakov in \cite{scott-sudakov} is,
    \begin{align*}
        c(G) = O \left(\frac{n}{2^{1-o(1)\sqrt{\log (n)}}} \right).
    \end{align*}
    \smallskip
    
    Along with these bounds based on the order of $G$, there exist similar bounds pertaining to the topology of $G$. Aigner and Fromme \cite{aigner-fromme} showed, rather surprisingly, that if a graph $G$ is planar, then $c(G) \leq 3$. Here, we present a natural extension of Aigner-Fromme's theorem, which to our knowledge has not appeared in the literature, to (potentially non-planar) graphs that can be covered by a planar graph. We say a graph $H$ \emph{covers} $G$ if there exists some surjection $\rho\colon V_H \rightarrow V_G$ that preserves adjacency. Clearly, if $H$ covers $G$, then $c(H) \geq C(G)$.
    \begin{remark}\label{cor:Archdeacon}
        If $G$ can be embedded in the projective plane, then $c(G) \leq 3$, since by Negami's theorem \cite{hlineny}, such graphs have planar covers.
    \end{remark}
    
    We note that this condition is equivalent to $G$ not having one of $35$ forbidden minors \cite{archdeacon-1}. Moreover, \cite{schroder} has shown that for genus-1 graphs, the cop number is at most four, while for genus-2 graphs the cop number is at most 5, which is the basis for \emph{Schr\"oder's conjecture}, which claims that
        the cop number of a genus-$g$ graph is at most $g + 3$.
   In the same paper, Schr\"oder showed $c(G) \leq \lfloor 3g/2 \rfloor + 3$, which has been improved by \cite{bowler}, where they showed that for a graph $G$ of genus $g$, we have $c(G) \leq (4g+10)/3$.
    
    We now shift our attention to chess graphs. Firstly, it can be shown trivially that rook graphs $\mathcal{R}_n$ are $2$-cop win for $n \geq 2$. Similarly, connected components of a bishop graph $\mathcal{B}_n$ are $2$-cop win for $n\geq 4$. King graphs $\mathcal{K}_n$ are cop-win as they are dismantlable. Knight graphs $\mathcal{N}_n$ are more interesting, and to our knowledge, cops and robbers on knight graphs has not been studied before in the literature. In \cref{sec:knight_graphs}, we prove the classification of the cop numbers for $\mathcal{N}_n$.

    \begin{theorem}\label{thm:cop knights}
         \begin{align*}
             c(\mathcal{N}_n) = \begin{cases}
                 1, \text{ for } 1\leq n \leq 2,\\
                 2, \text{ for } 4\leq n\leq 6, \\
                 3, \text{ otherwise}.
             \end{cases}
         \end{align*}
    \end{theorem}
    \medskip
    
    In \cref{sec:queen_and_royal_graphs}, we study cops and robbers on queen graphs $\mathcal{Q}_n$. The cop numbers $c(\mathcal{Q}_n)$ are known up to $n = 6$. It is also known that $c(\mathcal{Q}_n) \leq 4$ for all $n \geq 1$, as the cops simply guard all four of the robbers direction's of movement \cite{sullivan2017introduction}. In fact, Sullivan \textsl{et al.} showed in \cite{sullivan2017introduction} that $c(\mathcal{Q}_n) = 4$ for $n \geq 19$. We classify $c(\mathcal{Q}_n)$ for the remaining cases $7 \leq n \leq 18$.
    
    \begin{theorem}\label{thm: 3-cop win algo}
        $c(\mathcal{Q}_n) = 3$ for $7 \leq n \leq 18$.
    \end{theorem}
    
    This resolves the open problem of showing the monotonicity of $c(\mathcal{Q}_n)$ by determining the cop numbers of $\mathcal{Q}_n$ for all $n$ mentioned in \cite{sullivan2017introduction}. This is closely related to the open problem of showing the monotonicity of the domination number $\gamma(\mathcal{Q}_n)$, which has been shown to fail for rectangular queen graphs \cite{bozoki2019domination}. In particular, we find the smallest $4$-cop win queen graph to be $\mathcal{Q}_{19}$. Together, theorems \ref{thm:cop knights} and \ref{thm: 3-cop win algo} complete the classification of the cop number of all the $n \times n$ classical chess graphs.
    
    The next step is to look at generalizations of chess graphs. We call the first of these \emph{royal graphs}, which consist of a two-dimensional lattice of $n^2$ vertices, as a chessboard, as well as a set $D$ of directions in which we have incidences. Two vertices are connected if the direction between them is in $D$. For example, $\mathcal{Q}_n$ is a lattice of $n^2$ points, with the directions being $D=\{(1,0)$, $(0,1)$, $(1,1)$, $(1,-1)\}$. Furthermore, we call $\mathcal{Q}_n$ a \emph{4-royal graph}, since its set of directions, $D$, consists of four elements and we define \emph{$k$-royal graphs} similarly. Note that some royal graphs are disconnected. Since the cop number is additive on the connected components of $G$, we write $c(G)$ as the cop number of one of the connected components of $G$. At the end of \cref{sec:queen_and_royal_graphs}, we show that the cop number of $n\times n$ royal graphs with the same $D$, called a \emph{royal family}, is eventually constant. For $k$-royal graphs, the constant is exactly $k$.

    \begin{theorem}\label{thm: cop number bound for k-royal graphs}
        For any $k$-royal family $\{G_{n}\}_{n \in \mathbb{N}}$ of $k$-royal graphs, there exists $N \in \mathbb{N}$ such that $c(G_{n}) = k$ for all $n \geq N$.
    \end{theorem}
    \medskip

    Finally, we consider \emph{animal graphs}, which are defined similarly but two vertices $v, w$ share an edge if $v, w$ are the closest lattice points in some direction $d \in D$. For example, kings are $4$-animal, with the same direction set as $\mathcal{Q}_n$, and knights are also $4$-animal. We note that Theorem \ref{thm: cop number bound for k-royal graphs} is not true for animal graphs as $c(\mathcal{K}_n)=1$ and $c(\mathcal{N}_n)$ is eventually $3$. Aside from knight graphs, we give no results on animal graphs, but this seems to be a fruitful future direction as royal and animal graphs can model an object's movement given mechanical and environmental restrictions.
	\section{Knight graphs}\label{sec:knight_graphs}
	Here, we give~$c(\mathcal{N}_n)$ for all~$n \in \mathbb{N}$. The cases for $k=1,2$ are trivial, as these graphs have no edges. Additionally, we have $c(\mathcal{N}_3) = 3$ since $\mathcal{N}_3 \cong C_8 + K_1$. 

    \begin{lemma}\label{knight 4}
       $c(\mathcal{N}_4) = c(\mathcal{N}_5) = c(\mathcal{N}_6)= 2$.
    \end{lemma}
    \begin{proof}
        The winning strategy for the robber with one cop is to stay on a~$4$-cycle, which is a valid winning strategy forone1 cop on all knight graphs~$\mathcal{N}_n$ where~$n \geq 4$.
         For a $2$-cop win, for all cases, let the cops start at~$(3,3)$ and $(4,4)$. By brute force analysis, the cops win regardless of the robber's starting position.
    \end{proof}

    We now show that the cop number is three for $n\geq 7$. First, we give a $4$-cop win strategy, which gives a simpler version of the $3$-cop win strategy. Denote $N(v)$ as the \emph{neighborhood of $v$}, which is the set consisting of $v$ and its neighbors.
    \begin{lemma}
        $c(\mathcal{N}_n) \leq 4$ for all~$n \geq 7$.
    \end{lemma}
    \begin{proof}
        Consider a starting position for the cops at the four most central vertices. If~$n$ is odd, then one cop is at the central vertex and the other three are arranged so that the cops occupy a~$2 \times 2$ section of the board. Throughout the game, the four cops remain in this square formation. The union of the neighborhoods~$\bigcup_v N(v)$ for each of the cop vertices~$v$ forms a dense~$16$-vertex shape around the four vertices. 
        
        Regardless of where the robber starts, there exists a path that the formation can take (each cop moves in the same direction, maintaining the formation) to minimize the distance between the robber and the cops. If the robber moves away from the cop formation, the cops may \emph{mimic}, or move the same direction as the robber's move, so the distance will never increase. Notice that this cop formation guards a~$4\times 4$ subset of the board. Following this strategy ensures that the~$4 \times 4$ subset will eventually encapsulate the robber since if the robber were to continue moving outside of this region, it will eventually reach a boundary of the~$n \times n$ board such that it can no longer move away from the cops in at least one of the axes. 
        
        Once the robber is near the boundary, at worst, every other turn, the distance between the cops and robber is strictly decreasing. Thus, the cops win. 
    \end{proof}
    \smallskip
    
    \begin{proof}[Proof of Theorem \ref{thm:cop knights}]
        The winning strategy for the robber with two cops is to remain on an induced subgraph of $16$ vertices where each vertex has degree four (see~\cref{fig:knightpolytope}). Two cops can at most block all four neighbors of the robber, leaving the robber unthreatened, or one may threaten the robber on an adjacent vertex and the other may block two neighbors, leaving the robber with an escape route to its last neighboring vertex. 
    
        We show that three cops win via an algorithm. The cops create a formation, similar to the one from the proof of $c(\mathcal{N}_n) \leq 4$, that forces the robber onto the edge of the board and a losing position. The starting formation for the cops is a diagonal of length three. For example, they may start at~$(x - 1, y - 1), (x, y), (x + 1, y + 1)$ where~$(x, y)$ is a central vertex. Each starting vertex is of the same color, and the union of their neighborhoods covers a large, dense shape of only opposite-colored vertices. Notice that this cop formation can move in more directions than a single knight can. It can perform normal knight-like moves in the direction set $\{(\pm 1, \pm 2), (\pm2, \pm 1)\}$, and the additional moves in the direction set $\{(\pm 0, \pm 1), (\pm 1, \pm 0)\}$.
    
        The cops may chase the robber in the $3$-diagonal formation in a similar fashion to the $4$-cop win strategy such that the robber is eventually close enough for the cops to capture. In other words, we can assume that if the cops have not won, the robber is within the convex hull $S$ of the pink-colored squares of \cref{fig:knightdiag} or at the corner of the board, where the cops easily win.
    
        After the cops force the robber into $S$, casework shows that the cops can capture the robber, with some cases involving breaking the cop formation (see \cref{fig:3 cop ideal state}) to adhere to the boundaries of the board.
    \end{proof}
    \smallskip

    \begin{figure}[htbp!]
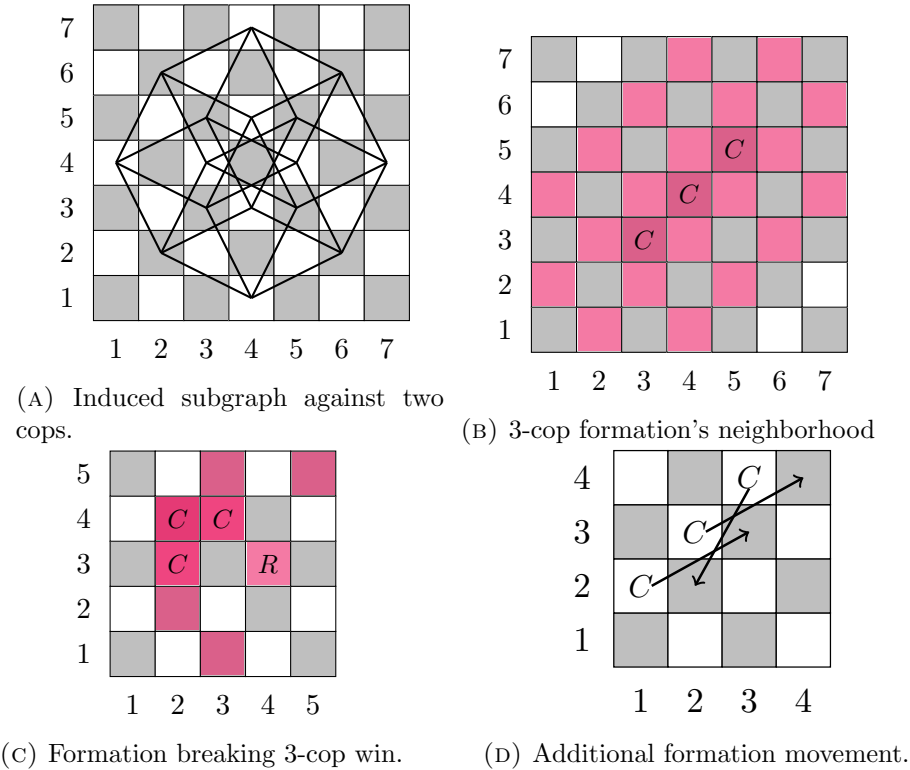

        \centering
        \begin{subfigure}[b]{0.45\textwidth}
        \centering
        \scalebox{1}{
        \drawchessboard{
          \draw[black, thick] (2,2) -- (4,1);
          \draw[black, thick] (6,2) -- (4,1);
          \draw[black, thick] (6,2) -- (7,4);
          \draw[black, thick] (6,6) -- (7,4);
          \draw[black, thick] (6,6) -- (4,7);
          \draw[black, thick] (2,6) -- (4,7);
          \draw[black, thick] (2,6) -- (1,4);
          \draw[black, thick] (2,2) -- (1,4);
          \draw[black, thick] (3,3) -- (4,5);
          \draw[black, thick] (4,5) -- (5,3);
          \draw[black, thick] (5,3) -- (3,4);
          \draw[black, thick] (3,4) -- (5,5);
          \draw[black, thick] (5,5) -- (4,3);
          \draw[black, thick] (4,3) -- (3,5);
          \draw[black, thick] (3,5) -- (5,4);
          \draw[black, thick] (5,4) -- (3,3);
          \draw[black, thick] (2,2) -- (3,4);
          \draw[black, thick] (2,2) -- (4,3);
          \draw[black, thick] (4,1) -- (3,3);
          \draw[black, thick] (4,1) -- (5,3);
          \draw[black, thick] (6,2) -- (4,3);
          \draw[black, thick] (6,2) -- (5,4);
          \draw[black, thick] (7,4) -- (5,3);
          \draw[black, thick] (7,4) -- (5,5);
          \draw[black, thick] (6,6) -- (4,5);
          \draw[black, thick] (6,6) -- (5,4);
          \draw[black, thick] (4,7) -- (3,5);
          \draw[black, thick] (4,7) -- (5,5);
          \draw[black, thick] (2,6) -- (3,4);
          \draw[black, thick] (2,6) -- (4,5);
          \draw[black, thick] (1,4) -- (3,5);
          \draw[black, thick] (1,4) -- (3,3);
        }{7}{true}
        }
        
        \caption{Induced subgraph against two cops.}\label{fig:knightpolytope}
        \end{subfigure}
        \begin{subfigure}[b]{0.45\textwidth}
        \centering
        \scalebox{1}{
        \drawchessboard{
        \foreach \dx/\dy in {2/1, 2/-1, -2/1, -2/-1, 1/2, 1/-2, -1/2, -1/-2, 0/0} {
            \pgfmathtruncatemacro{\cx}{5 + \dx}
            \pgfmathtruncatemacro{\cy}{5 + \dy}
            \ifthenelse{\cx > 0 \AND \cx < 8 \AND \cy > 0 \AND \cy < 8}{
                \node[fill=WildStrawberry, opacity=0.6, minimum size=.57cm, inner sep=0] at (\cx, \cy) {};
            }{};
        }
        \foreach \dx/\dy in {2/1, 2/-1, -2/1, -2/-1, 1/2, 1/-2, -1/2, -1/-2, 0/0} {
            \pgfmathtruncatemacro{\cx}{3 + \dx}
            \pgfmathtruncatemacro{\cy}{3 + \dy}
            \ifthenelse{\cx > 0 \AND \cx < 8 \AND \cy > 0 \AND \cy < 8}{
                \node[fill=WildStrawberry, opacity=0.6, minimum size=.57cm, inner sep=0] at (\cx, \cy) {};
            }{};
        }
        \node[fill=WildStrawberry, opacity=0.6, minimum size=.57cm, inner sep=0] at (4,4) {};
        \node[fill=WildStrawberry, opacity=0.6, minimum size=.57cm, inner sep=0] at (2,3) {};    \node[fill=WildStrawberry, opacity=0.6, minimum size=.57cm, inner sep=0] at (3,2) {};    \node[fill=WildStrawberry, opacity=0.6, minimum size=.57cm, inner sep=0] at (6,5) {};    \node[fill=WildStrawberry, opacity=0.6, minimum size=.57cm, inner sep=0] at (5,6) {};
        \node[chesspiece] at (5,5) {\small{$C$}};
        \node[chesspiece] at (3,3) {\small{$C$}};
        \node[chesspiece] at (4,4) {\small{$C$}};
    }{7}{true}
        }
        \caption{$3$-cop formation's neighborhood}\label{fig:knightdiag}
        \end{subfigure}

        \centering
        \begin{subfigure}[b]{0.48\textwidth}
        \centering
        
        \drawchessboard{
        \foreach \dx/\dy in {1/1, 1/2, 1/3, 2/0, 2/3, 2/4, 4/4} {
            \pgfmathtruncatemacro{\cx}{1 + \dx}
            \pgfmathtruncatemacro{\cy}{1 + \dy}
            \ifthenelse{\cx > 0 \AND \cx < 8 \AND \cy > 0 \AND \cy < 8}{
                \node[fill=WildStrawberry, opacity=0.6, minimum size=.57cm, inner sep=0] at (\cx, \cy) {};
            }{};
        }
        \node[fill=WildStrawberry, opacity=0.6, minimum size=.57cm, inner sep=0] at (2,3) {};
        \node[fill=WildStrawberry, opacity=0.6, minimum size=.57cm, inner sep=0] at (2,4) {}; 
        \node[fill=WildStrawberry, opacity=0.6, minimum size=.57cm, inner sep=0] at (3,4) {};  
        \node[fill=WildStrawberry, opacity=0.6, minimum size=.57cm, inner sep=0] at (4,3) {};  
        \node[chesspiece] at (2,3) {\small{$C$}};
        \node[chesspiece] at (2,4) {\small{$C$}};
        \node[chesspiece] at (3,4) {\small{$C$}};
        \node[chesspiece] at (4,3) {\small{$R$}};
    }{5}{true}
        \caption{Formation breaking $3$-cop win.}\label{fig:3 cop ideal state}
        \end{subfigure}
        \hfill
        \begin{subfigure}[b]{0.48\textwidth}
        \centering
        \scalebox{1.2}{
          \drawchessboard{
        \node[chesspiece] at (1,2) {\small{$C$}};
        \node[chesspiece] at (2,3) {\small{$C$}};
        \node[chesspiece] at (3,4) {\small{$C$}};
        \draw[->, thick, black] (1.2,2) -- (3,3);
        \draw[->, thick, black] (2.2,3) -- (4,4);
        \draw[->, thick, black] (3,3.8) -- (2,2);
    }{4}{true}
    }
        \caption{Additional formation movement.}\label{fig:3 diag movement}
        \end{subfigure}
        
        \caption{Knight strategies}
        \label{fig:knightstuff2}
    \end{figure}
	\section{Queen and royal graphs}\label{sec:queen_and_royal_graphs}
	\subsection{Queen graphs}\label{sub:queen_graphs}
	We begin by defining greedy strategies for both the cops and the robber and show that these strategies are at Nash equilibrium and lead to a~$3$-cop win on~$\mathcal{Q}_n$ for~$7 \leq n \leq 18$. At any time (turn)~$t$, the robber has a shortest diagonal~$s_t$, the length of which we denote as~$\Phi(s_t)$. Let~$f_t(c, r) = \Phi(s_{t + 1})$ where~$s_{t + 1}$ is the short diagonal obtained after a pair of moves~$(c, r)$ by the cops and robber, respectively. Let our greedy strategies be defined as
	 \begin{align*}
	c^\star &= \text{arg} \min_{c}\max_{r}{f_t(c, r)}, \\
    r^\star &= \text{arg} \max_{r}\min\limits_{c}{f_{t+1}(c,r)}.
	 \end{align*}
	Let~$c_{i, t}$ be the position of the~$i$-th cop at time~$t$ and every possible cop position~$c_{i,t} \in C_{i,t}$ coincides exactly with the elements of~$N(c_{i, t - 1})$, and the robber's possible moves~$r$ at time~$t$ coincide with elements of~$R_t= N(r_{t-1}) \setminus \cup_{i = 1}^{k} \{c_{i, t}\}$. In words, the cops want to minimize the length of the robber's short diagonal, while the robber wants to maximize it. Both players calculate one move in advance.

	\begin{lemma}\label{monotonic}
		For all states at the beginning of timestep~$t$ (\textsl{i.e.}, beginning of the cops' move), there exists a sequence of moves for three cops on~$\mathcal{Q}_n$ such that 
		$\Phi(s_{t+1}) \leq \Phi(s_t)$,
		is guaranteed. In other words, under this cop strategy,~$\Phi$ is non-increasing.
	\end{lemma}
	\begin{proof}
		Observe that with three cops, three axes of movement can be directly guarded, and thus the cops can force the robber to a single axis of movement indefinitely. Choose to leave~$s_t$ available for robber movement. This shows that~$\Phi(s_t)$ can never be larger than~$\Phi(s_{t-1})$.
	\end{proof}
	 
	\begin{lemma} \label{saddle}
    	Given the greedy strategies~$c^\star, r^\star$, for all~$r$,
    	\begin{align*}
    	    f(c^\star, r) \leq f(c^\star, r^\star).
    	\end{align*}
    	In other words, the robber does not benefit from changing their strategy, given that the opposing greedy cops' strategy is in play.
	\end{lemma}
	\begin{proof} 
		By Lemma \ref{monotonic}, three cops can control~$\Phi(s_t)$. Thus, if the robber chooses a suboptimal move~$r$ at time $T$, then~$f_t(c^\star, r^\star)\leq f_T(c^\star, r)$ for all~$t>T$. In other words, the short diagonal length of the suboptimal state will bound all future short diagonal lengths from above. Thus, we can say that 
		\begin{align*}
			f_t(c^\star, r) \leq \max_{r \in R_t} f_t(c^\star,r) = f_t(c^\star, r^\star),
		\end{align*}
		by definition, without worry that a suboptimal robber move may lead to a better position at a future turn.
	\end{proof}
    \smallskip
	
	\begin{lemma}\label{alg}
	   Under the greedy strategies for both teams, we have that three cops can capture the robber in a finite number of turns for~$n\leq 18$.
	\end{lemma}
	\begin{proof}
		We show this by direct computation. We first rank moves by short diagonal length, then tiebreak with the number of available squares left for the robber. More formally, we tiebreak moves as follows:
		\begin{align*}
			c^\star = \text{arg} \min_{c}\max_{r}{|R_t|},\qquad r^\star= \text{arg} \max_{r}\min_{c}{|R_{t+1}|}.
		\end{align*}
		This tiebreaker is precisely how the cops eventually capture the robber after the robber is pushed to a sufficiently small~$s_t$.
		
		For~$n \leq 17$, three cops capture the robber in under~$30$ turns. For~$\mathcal{Q}_{18}$, the cops win in at most~$151$ turns\footnote{These algorithms are implemented using \texttt{SageMath}. This code is open source and can be found at \href{https://github.com/danielma4/Cops_and_Robbers}{\url{github.com/danielma4/Cops_and_Robbers}}.}.
		
		Note that this algorithm does not rely on randomness in ranking cop or robber moves. No states are repeated, so cops employ brute force in trying all moves that minimize~$\Phi$ and~$|R_t|$ until one eventually pushes the robber to a suboptimal~$s_t$.
	\end{proof}
    \smallskip
    
	We've shown that three cops are sufficient to capture the robber on $\mathcal{Q}_n$ for $n \leq 18$. It remains to show that the robbers win against two cops.
	\begin{lemma}[\cite{sullivan2017introduction}]\label{lem:guarding}
	   On $\mathcal{Q}_n$, a cop can guard at most three vertices on each of the robber’s lines. If the cop also threatens, then it can guard at most two vertices on another line.
	\end{lemma}
	\begin{proof}[Proof of Theorem~\ref{thm: 3-cop win algo}]
		We begin by showing that two cops cannot win on~$\mathcal{Q}_n$ for~$7\leq n \leq 18$. To win against the robber, we have two cops directly threaten on two unique directions of the robber's movement and indirectly guard all vertices on the last two lines. This ensures maximal coverage by the cops by Lemma~\ref{lem:guarding}. Note that out of the four directions of movement, the horizontal and vertical axes will always be of length~$n$. Thus, the cops will choose to occupy the robber's horizontal and vertical axes of movement, leaving the two diagonals available, as these diagonals can decrease in length as the robber moves away from the center of the graph. Call these diagonals~$L_1$ and $L_2$. The cops need to guard all vertices on both of these axes, so we concern ourselves with the longest of these diagonals. Thus, we examine the case for which indirectly guarding is easiest for the cops, \textsl{i.e.}, guarding the diagonal line~$L$ with length~$\min\max\{(\mathrm{length}(L_1), \mathrm{length}(L_2))\}$. 
	
		Note that the length of $L$ is minimized at the middle of an edge of the board, when the robber has equal length diagonals, or diagonals which differ in length by one. The length of~$L$ is equal to~$n - \lfloor \frac{n - 1}{2} \rfloor$.
        By Lemma \ref{lem:guarding}, two cops can only cover five vertices on this diagonal. Thus, when~$n - \lfloor \frac{n - 1}{2} \rfloor > 5$, two cops cannot possibly cover this entire axis. It can be verified that~$n = 10$ is the first value for which this inequality is satisfied. Thus, two cops cannot win on any~$\mathcal{Q}_n$ for which~$n \geq 10$.
	
		For~$\mathcal{Q}_n$ with $n = 7, 8, 9$, we consider all possible cases, of which there are $3 \cdot \binom{n^2}{3}$. We use brute force computation to conclude that for~$n = 7, 8, 9$ there does not exist a~$2$-cop win state. Lastly for~$7\leq n \leq 18$, we apply lemmas \ref{saddle} and \ref{alg} to show the claim.
	\end{proof}
	\subsection{Royal Graphs}\label{sub:royal_graphs}
	Here, we show Theorem \ref{thm: cop number bound for k-royal graphs}. First, we give a weaker version of a result of \cite{sullivan2017introduction}, which says $c(\mathcal{Q}_n) =4$ for $n \geq 19$.
\begin{lemma}\label{cop number bound for queens}
	For $n \geq 22$, we have $c(\mathcal{Q}_n) = 4$.
\end{lemma}
\begin{proof}
		We present a winning strategy for the robber against three cops on $\mathcal{Q}_n$ with $n \geq 22$. Consider an octagon chosen such that each side has eight vertices. By Lemma \ref{lem:guarding}, three cops can cover six vertices when each threatens on a unique axis. With the addition of the robber's current vertex, the cops can cover at most seven vertices on any axis of the robber's movement. Notice that if we restrict the robber's movement to this octagon-shaped induced subgraph, all four axes of queen movement are always at least of length eight. Thus, there will always be an axis of movement for the robber that has at least one unguarded vertex to move to. The smallest queen graph on which this octagon is a subgraph is $\mathcal{Q}_{22}$. As all $\mathcal{Q}_n$ are $4$-cop win, we conclude that $c(\mathcal{Q}_n) = 4$ for all $n \geq 22$.
	\end{proof}
    \smallskip
	
	We call this argument the \emph{octagon argument}, and generalize it to show \cref{thm: cop number bound for k-royal graphs}. Recall that a \emph{royal graph} is a graph $([n]^2, E_d)$ and a set $D$ of directions such that any two points $x, y \in [n]^2$ are adjacent if and only if the slope of the line connecting them is an element of $D$.
	\begin{proof}[Proof for Theorem \ref{thm: cop number bound for k-royal graphs}]
		Fix a set $D = \{d_1, \dots, d_k\}$ of $k$-many directions, and let $\{G_n\}_{n \in \mathbb{N}}$ be a $k$-royal family of graphs. If $c$ is a cop that is not attacking the robber $r$, then each line of $c$ intersects each line (except that which is parallel to a given line of $r$) of $r$ in at most one point. In this case, we note that $c, r$ have at most $(k - 1)$-many common neighbors through a given line of $r$. On the other hand, if $c$ is attacking $r$ on a line with slope $d_j$ for some $j \in [k]$, then both $c$ and $r$ have at most $(k - 2)$-many common neighbors through a line of slope $d_{\ell} \neq d_j$ of $r$ that do not intersect a line of slope $d_\ell$. This is an analogue of Lemma \ref{lem:guarding}
	
		Clearly, $k$-many cops win by attacking $r$ on each of its $k$-many lines. We now show that $(k-1)$-many cops do not win for large boards. For large $n$, the number of available squares on a given line grows linearly, yet the number that $(k-1)$-many cops can cover while attacking is bounded above by a constant, namely $(k - 1)(k - 2) + 1$. There is an analogous octagon argument for this family of graphs. Let $L_j$ be the collection of lines with slope $d_j$ with more than $(k - 1)(k - 2) + 1$ vertices, and let $N \in \mathbb{N}$ be large enough so that $|L_j| \geq 2$ for all $j \in [k]$. This defines a convex region of the board. Indeed, for $n$ large, the intersection of the regions defined by the $L_j$ has non-zero measure and gives a space in which the robber may move without being captured by $(k - 1)$-many cops.
	\end{proof}
    \section*{About the authors}
As stated previously, this research was part of the NSMRP REU.

\noindent Undergraduate mentees:
\begin{itemize}[leftmargin=10pt,itemindent=1em]
    \item Sally Ambrose, \url{ambrose.sa@northeastern.edu}
    \item Jacob Chen, \url{chen.jacob3@northeastern.edu}
    \item Daniel Ma, \url{ma.dan@northeastern.edu}
    \item Wraven Watanabe, \url{watanabe.so@northeastern.edu}
    \item Stephen Whitcomb, \url{whitcomb.s@northeastern.edu}
    \item Shanghao Wu, \url{wu.shangh@northeastern.edu}
\end{itemize}
PhD student mentors:
\begin{itemize}[leftmargin=10pt,itemindent=1em]
    \item Evan Angelone, \url{griggs.e@northeastern.edu}
    \item Arturo Ortiz San Miguel, \url{ortizsanmiguel.a@northeastern.edu}
\end{itemize}
	\medskip
	\printbibliography
\end{document}